\documentclass[12pt, reqno]{amsart}
\usepackage{latexsym,amsfonts, amssymb,amsthm,curvesls,curves,amsmath,epic,eepic,graphicx,cancel,color,hyperref}

\newtheorem{theorem}{Theorem}[section]
\newtheorem{lemma}[theorem]{Lemma}

\newtheorem{definition}{Definition}[section]

\newcommand{\be}{\begin{equation}}
\newcommand{\ee}{\end{equation}}
\newcommand{\ben}{\begin{equation*}}
\newcommand{\een}{\end{equation*}}
\newcommand{\bea}{\begin{eqnarray}}
\newcommand{\eea}{\end{eqnarray}}
\newcommand{\bean}{\begin{eqnarray*}}
\newcommand{\eean}{\end{eqnarray*}}

\allowdisplaybreaks

\title[From Dominated Splitting to Hyperbolicity]{Characterizations of Dominated Splitting System and its  Relation to Hyperbolicity}
\author{Chun Fang}
\author{Mats Gyllenberg}
\author{Shitao Liu}

\address{Depatment of Mathematics and Statistics, University of Helsinki, P.O. Box 68 FI-00014}
\email{chun.fang@helsinki.fi}
\email{mats.gyllenberg@helsinki.fi}
\email{shitao.liu@helsinki.fi}
\date{\today}

\begin{document}
\maketitle

\begin{abstract}
Hyperbolicity and dominated splitting are two of the most important  concepts in the global analysis of differentiable dynamics. In this paper we give several equivalent characterizations of the dominated splitting and in particular we show a criterion for  dynamical systems being dominated splitting in terms of hyperbolicity.
\end{abstract}
\section{Introduction}
\setcounter{equation}{0}

In this paper we study dominated splitting and its relation to hyperbolicity. It is well-known that hyperbolicity and dominated splitting are two of the most important concepts during the global analysis of differentialble dynamics. Hyperbolicity, as the cornerstone of uniform and robust chaotic dynamics, has been fairly well studied both from topological and the statistical points of view during the past several decades. In particular, the Spectral Decomposition Theorem built up in \cite{N,S2} for hyperbolic systems completely describes the dynamics of these systems. Moreover, hyperbolicity also showed to be the key ingredient in characterization of structural stability, together with a transversality condition. Here, a hyperbolic dynamical system means a system such that its limit set (the minimum closed invariant set that contains the $\omega$- and $\alpha$-limit set of any orbit) is a hyperbolic set (see definition below). Unfortunately, it is soon realized by Smale \cite{S} and others that hyperbolicity is not typical in the sense that it is not dense in the space of all $C^r$ differentiable diffeomorphisms on the manifold. In order to improve this, many efforts have been done and devoted to expand this notion to involve a larger class of dynamics. One important relaxation to hyperbolicity is dominated splitting into two invariant complementary subbundles, which was introduced independently by Ma\~{n}\'{e} \cite{M}, Liao \cite{L} and Pliss \cite{VP} in the context of the stability conjecture, by letting one of them contracted or expanded exponentially faster than the other under the iterations. It is not difficult to see that normally hyperbolic closed invariant curves with dynamics conjugated to irrational rotations admits a dominated splitting but it is not hyperbolic. See also \cite{FGW,SY} for more examples from application fields that satisfy domination but not hyperbolic. We also refer to \cite{BG} and the review paper \cite{Pujals} for some characterizations of dominated splitting and interesting materials from hyperbolcity to dominated splitting. Meanwhile, one must recognize that the theory of dominated splitting is far from complete and successful than that of hyperbolicity. Surprisingly, for $C^2$ diffeomorphisms on the compact surface and for three-dimensional flows the remarkable works of Pujals and Sambarino \cite{PS,PS1} and work of Arroyo and Hertz \cite{AH} give a satisfactory description of the dynamics of any compact invariant set having dominated splitting. In particular, they show a similar Spectral Decomposition Theorem for the limit set of a dynamics with the assumption of dominated splitting.

With these results in mind, a natural question one may ask is what kind of relationship between these two types of dynamical systems (besides the natural implication of dominated splitting by hyperbolicity). Motivated by \cite{Co,P}, in the present paper we will give a series of equivalent characterizations for dominated splitting which ultimately lead to a criterion for dynamical systems being dominated splitting over some invariant set (see definition below) in terms of hyperbolicity. That is, we will show that the difference between dominated splitting and hyperbolicity is only a functional torsion.

We start with first recalling some basic definitions and notations.

Let $M$ be a closed $d$-dimensional manifold and $f:M\rightarrow M$ be a diffeomorphism. Let $\mathcal{E}$ be a continuous $d$-dimensional vector bundle over $M$ with a continuous inner product $\langle\cdot,\cdot\rangle$ and let $\|\cdot\|$ be the induced norm. A \emph{linear cocycle} over a dynamical system $(M,f)$ is an automorphism $F$ of a vector bundle $\mathcal{E}$ over $M$ that projects to $f$. In particular we consider in this paper the trivial vector bundle $M\times\mathbb{R}^d$, then any linear cocycle $F: M\times\mathbb{R}^d\rightarrow M\times\mathbb{R}^d$ can be identified with a matrix-valued function $A:M\rightarrow GL(d,\mathbb{R})$ var $F(x,v)=(f(x),A(x)v)$. For simplicity, we denote this linear cocycle with a pair $(f,A)$.

A \emph{projector} on $M\times\mathbb{R}^n$ is defined to be a continuous mapping $P:M\times\mathbb{R}^n\rightarrow M\times\mathbb{R}^n$ such that $P$ maps each fiber to itself and $P$ is a projection on each fiber. In other words, one has
\[P(x,u)=(x,P(x)u),\]
where $P(x)$ is a projection.

\begin{definition}\label{hyperbolicity}
An $f$-invariant  compact set $\Lambda\subset M$ is said to be \emph{hyperbolic} if the tangent bundle $T_{\Lambda}M$ over $\Lambda$ admits a continuous decomposition
$$T_{\Lambda}M=E^s\oplus E^u,$$
invariant under the derivative $Df$ and there exist positive constants $C>1$ and $\alpha>0$ such that
\begin{eqnarray*}
&&\|Df^{n}(x)Q(x)\|\leq Ce^{- n\alpha},\\[3mm]
&&\|(I-Q(x))Df^{-n}(f^{n}(x))\|\leq Ce^{-n\alpha},
\end{eqnarray*}
for all $n\geq 1$, and $x\in\Lambda$. Here $Q$ and $I-Q$ are the oblique projectors corresponding to decomposition $T_{\Lambda}M=E^s\oplus E^u$.
\end{definition}

\begin{definition}\label{dominatesplitting}
Let $\Lambda\subset M$ be a compact $f$-invariant set, $k\geq 2$. An $\it{Df}$-invariant splitting $T_xM=E_1(x)\oplus\cdots\oplus E_k(x)$, $x\in\Lambda$ of the tangent bundle over $\Lambda$ is called \emph{dominated} if there exist positive constants $C>1$ and $\alpha>0$ such that for every $i<j$, every $x\in\Lambda$ and every pair of vectors $u, v\in T_xM$, we have $\forall m,n\in\mathbb{Z}$, $m\geq 0$,
\begin{equation}\label{s1eq4}
\frac{\|Df^{n+m}(x)P_i(x)u\|\|Df^n(x)P_j(x)v\|}{\|Df^n(x)P_i(x)u\|\|Df^{n+m}(x)P_j(x)v\|}
\leq Ce^{-m\alpha}
\end{equation}
for all $m,n\in\mathbb{Z}$ with $m\geq 0$, provided the denominator of \eqref{s1eq4} do not vanish, where $P_1,\cdots,P_k$ are the oblique projectors with respect to decomposition $T_{\Lambda}M=E_1\oplus\cdots\oplus E_k$.
\end{definition}

For convenience, we write $E_i\prec E_j$ if \eqref{s1eq4} holds.
Let $n_i=\dim(E_i)$ and suppose above splitting exist, we say that the dynamical system $f:M\rightarrow M$ admits a $(n_1,\cdots,n_k)$-dominated splitting over $\Lambda$ with decomposition $T_{\Lambda}M=E_1\oplus\cdots\oplus E_k$.

As we mentioned before, it is easy to see from the definitions that hyperbolicity implies dominated splitting. In the following sections, we will investigate various properties and characterizations of dominated splitting systems and their difference with hyperbolic systems. The rest of the paper is organized as follows. In section 2 we show that dominated splitting systems are reducible systems, and use this result to give the first equivalent characterization of dominated splitting. Then using this result we give in section 3 another equivalent description of dominated splitting in terms of upper and lower functions. Finally, section 4 establishes the third equivalent characterization of dominated splitting via summably separated functions which says that the difference between dominated splitting and hyperbolicity is a function torsion.

\section{Reducible systems}
In this section we introduce and study reducible systems and prove it is a necessary condition for dominated splitting. Moreover, we will show dominated splitting can be characterized by a simple inequality.

\begin{definition}\label{reducible}
Let $T_{\Lambda}M=E_1\oplus\cdots\oplus E_k$, $k\geq 2$ be an $\it{Df}$-invariant splitting of tangent bundle over $\Lambda$ and $P_1,\cdots,P_k$ be the corresponding oblique projectors (not necessarily continuous). Dynamical system $f$ is said to be \emph{reducible} with respect to this decomposition if there exists constant $K>0$ such that
\be\label{s1eq7}
\|Df^n(x)P_i(x)Df^{-n}(f^n(x))\|\leq K, \quad i=1,\cdots,k,
\ee
for all $n\in\mathbb{Z}$ and $x\in\Lambda$.
\end{definition}



The next lemma states that dominated splitting implies reducible.

\begin{lemma}\label{lemma2}
If dynamical system $f:M\rightarrow M$ admits $(n_1,\cdots,n_k)$-dominated splitting over $\Lambda$ with decomposition $T_{\Lambda}M=E_1\oplus\cdots\oplus E_k$, then it is reducible with respect to the same decomposition.
\end{lemma}
\begin{proof}
Fix $1\leq i\leq k-1$. Since $E_1\prec\cdots\prec E_k$ and use the clustering property of dominated splitting (see \cite[Appendix B]{BDV}), we have $(E_1\oplus\cdots\oplus E_i)\prec(E_{i+1}\oplus\cdots\oplus E_k)$ with projectors $Q_1=P_1\oplus\cdots\oplus P_i$, $Q_2=P_{i+1}\oplus\cdots\oplus P_{k}$. By definition, we can find constants $C>1$ and $\alpha>0$ such that $\forall m,n\in\mathbb{Z}$, $m\geq 0$,
\begin{equation}\label{eq2}
\frac{\|Df^{n+m}(x)Q_1(x)u\|\|Df^n(x)Q_2(x)v\|}{\|Df^n(x)Q_1(x)u\|\|Df^{n+m}(x)Q_2(x)v\|}
\leq Ce^{-m\alpha},
\end{equation}
where $u,v\in T_xM$ such that the dominators do not vanish. In what follows, fix $0<\beta<1$ such that $\beta<\|Df\|<\frac{1}{\beta}$, where $\|Df\|=\sup_{x\in\Lambda}\|Df(x)\|$.

Consider the term $\|\frac{Df^{n+m}(x)Q_1(x)u}{\|Df^n(x)Q_1(x)u\|}+\frac{Df^{n+m}(x)Q_2(x)v}{\|Df^n(x)Q_2(x)v\|}\|$. On the one hand, by \eqref{eq2} we have
\bea\label{s1eq8}
&&\left\|\frac{Df^{n+m}(x)Q_1(x)u}{\|Df^n(x)Q_1(x)u\|}+\frac{Df^{n+m}(x)Q_2(x)v}
{\|Df^n(x)Q_2(x)v\|}\right\|\nonumber\\[3mm]
& \geq & \frac{\|Df^{n+m}(x)Q_1(x) u\|}{\|Df^n(x)Q_1(x)u\|}\left(\frac{\|Df^{n}(x)Q_1(x)u\|\|Df^{n+m}(x)Q_2(x)v\|}
{\|Df^{n+m}(x)Q_1(x)u\|\|Df^n(x)Q_2(x)v\|}-1\right) \nonumber \\[3mm]
& \geq & \frac{\|Df^{n+m}(x)Q_1(x)u\|}{\|Df^n(x)Q_1(x)u\|}\left(C^{-1}e^{m\alpha}-1\right) \nonumber \\[3mm]
& \geq & \beta^m (C^{-1}e^{m\alpha}-1),
\eea
where we fix $m$ large enough so that $C^{-1}e^{m\alpha}-1 >0$.

On the other hand,
\bea\label{s1eq9}
&&\left\|\frac{Df^{n+m}(x)Q_1(x)u}{\|Df^n(x)Q_1(x)u\|}+\frac{Df^{n+m}(x)Q_2(x)v}
{\|Df^n(x)Q_2(x)v\|}\right\|\nonumber\\[3mm]
& = & \left\|Df^{m}(f^n(x))\left(\frac{Df^{n}(x)Q_1(x)u}{\|Df^n(x)Q_1(x)u\|}+
\frac{Df^{n}(x)Q_2(x)v}{\|Df^n(x)Q_2(x)v\|}\right)\right\| \nonumber \\[3mm]
& \leq & \beta^{-m}\left\|\frac{Df^{n}(x)Q_1(x)u}{\|Df^n(x)Q_1(x)u\|}+
\frac{Df^{n}(x)Q_2(x)v}{\|Df^n(x)Q_2(x)v\|}\right\|
\eea

Combine \eqref{s1eq8} and \eqref{s1eq9} we readily get
\be\label{s1eq10}
\left\|\frac{Df^{n}(x)Q_1(x)u}{\|Df^n(x)Q_1(x)u\|}+\frac{Df^{n}(x)Q_2(x)v}
{\|Df^n(x)Q_2(x)v\|}\right\| \geq \beta^{2m}(C^{-1}e^{m\alpha}-1)
\ee

Using the inequality $\|x\|\|\frac{x}{\|x\|}+\frac{y}{\|y\|}\|\leq 2\|x+y\|$, for any $x, y\in \mathbb{R}^d$, $d\geq 1$ with $x=Df^n(x)Q_1(x)u$ and $y=Df^n(x)Q_2(x)v$, we then have
\begin{multline*}\label{s1eq11}
\|Df^n(x)Q_1(x)u\|, \|Df^n(x)Q_2(x)v\| \\[3mm]
\leq 2\beta^{-2m}(C^{-1}e^{m\alpha}-1)^{-1}\|Df^n(x)Q_1(x)u+Df^n(x)Q_2(x)v\|.
\end{multline*}

It implies, by letting $u=v$, that
\begin{eqnarray*}\label{s1eq12}
&& \|Df^n(x)Q_1(x)u\| \nonumber \\[3mm]
&\leq& 2\beta^{-2m}(C^{-1}e^{m\alpha}-1)^{-1} \|Df^n(x)Q_1(x)u+Df^n(x)Q_2(x)u\|\nonumber\\[3mm]
&=& 2\beta^{-2m}(C^{-1}e^{m\alpha}-1)^{-1} \|Df^n(x)u\|.
\end{eqnarray*}

Replaced $u$ by $Df^{-n}(f^n(x))u$ we get from above that
\be\label{s1eq13}
\|Df^n(x)(P_1(x)\oplus\cdots\oplus P_i(x))Df^{-n}(f^n(x))\|\leq 2\alpha^{-2m}(C^{-1}e^{m\alpha}-1)^{-1}
\ee

Note that \eqref{s1eq13} holds for all $i=1,\cdots,k-1$ and $Df^n(x)(P_1(x)\oplus\cdots\oplus P_k(x))Df^{-n}(x)\equiv I$, therefore \eqref{s1eq7} holds for $i=1,\cdots,k$.
\end{proof}

Now we can prove our first characterization of dominated splitting.

\begin{theorem}\label{mainthm1}
Dynamical system $f:M\rightarrow M$ admits a $(n_1,\cdots,n_k)$-dominated splitting over invariant set $\Lambda$ if and only if there exist supplementary $\it{Df}$-invariant projectors $P_1,\cdots,P_k$ on $T_{\Lambda}M$ with $\dim(P_i)=n_i$ and constants $K>1$ and $\alpha>0$ such that for $i=1\cdots,k-1$,
\begin{multline}\label{eq1}
\|Df^{n+m}(x)P_i(x)Df^{-n}(f^n(x))\| \\
\|Df^n(x)P_{i+1}(x)Df^{-(n+m)}(f^{n+m}(x))\|\leq Ke^{-m\alpha},
\end{multline}
hold for all $m,n\in\mathbb{Z}$, $m\geq 1$ and $x\in\Lambda$.
\end{theorem}

\begin{proof}
Suppose the splitting $T_{x}M=E_1(x)\oplus\cdots\oplus E_k(x)$, $x\in\Lambda$ is $(n_1,\cdots,n_k)$-dominated and let $P_1(x),\cdots,P_k(x)$, $x\in\Lambda$ be corresponding projectors. We show that these projectors satisfy condition \eqref{eq1} and hence complete the proof of necessity.

To do this, let us fix $1\leq i\leq k-1$, $x\in\Lambda$ and choose any $u,v\in T_xM$ with $P_i(x)u\neq 0$, $P_{i+1}(x)v\neq 0$. By \eqref{s1eq4}, there exist constants $C> 1$ and $\alpha>0$ such that
\begin{equation*}\label{s1eq14}
\frac{\|Df^{n+m}(x)P_i(x)u\|\|Df^n(x)P_{i+1}(x)v\|}{\|Df^n(x)(x)P_i(x)u
\|\|Df^{n+m}(x)P_{i+1}v\|}\leq Ce^{-m\alpha}
\end{equation*}

After replacing $u$ with $Df^{-n}(f^n(x))u$ and $v$ with $Df^{-(n+m)}(f^{n+m}(x))v$, we get
\begin{multline*}
\|Df^{n+m}(x)P_i(x)Df^{-n}(f^n(x))u\|\|Df^n(x)P_{i+1}(x)Df^{-(n+m)}(f^{n+m}(x))v\|\\[3mm]
\leq Ce^{-m\alpha}\|Df^n(x)P_i(x)Df^{-n}(f^n(x))u\| \\ \
|Df^{n+m}(x)P_{i+1}(x)Df^{-(n+m)}(f^{n+m}(x))v\|
\leq CK^2e^{-m\alpha}=\tilde{K}e^{-m\alpha}.
\end{multline*}

Note that here $\tilde{K}=C K^2$ and $\alpha$ are independent of $x\in\Lambda$ and above result holds for all $x\in\Lambda$. This proves our claim.

Conversely, suppose \eqref{eq1} holds. Namely, $\forall u\in T_{f^n(x)}M$, $v\in T_{f^{n+m}(x)}M$, we have
\begin{multline}\label{s1eq16}
\|Df^{n+m}(x)P_i(x)Df^{-n}(f^n(x))u\| \\
\|Df^n(x)P_{i+1}(x)Df^{-(n+m)}(f^{n+m}(x))v\|\leq Ke^{-m\alpha}\|u\|\|v\|
\end{multline}

Without loss of generality, we assume $P_i(x)Df^{-n}(f^n(x))u\neq 0$ and $P_{i+1}(x)Df^{-(n+m)}(f^{n+m}(x))v\neq 0$. By replacing $u$ with $Df^n(x)P_i(x)u$ and $v$ with $Df^{n+m}(x)P_{i+1}(x)v$ in \eqref{s1eq16}, we obtain for $i=1,\cdots,k-1$,
\begin{equation}\label{eq3}
\frac{\|Df^{n+m}(x)P_i(x)u\|\|Df^n(x)P_{i+1}(x)v\|}{\|Df^{n}(x)P_i(x)u\|
\|Df^{n+m}(x)P_{i+1}(x)v\|}\leq Ke^{-m\alpha}.
\end{equation}

Hence \eqref{s1eq4} holds for all pair $(i,j)$, where $1\leq i\leq k-1$ and $j=i+1$. Fix $i$, for $j>i+1$ one can induce through \eqref{eq3} that
\begin{equation*}
\frac{\|Df^{n+m}(x)P_i(x)u\|\|Df^n(x)P_{j}(x)v\|}{\|Df^{n}(x)P_i(x)u\|
\|Df^{n+m}(x)P_{j}(x)v\|}\leq (Ke^{-m\alpha})^{j-i}\leq K^ne^{-m\alpha},
\end{equation*}

Thus we proved that \eqref{s1eq4} holds for all $1\leq i<j\leq k$.
\end{proof}

\begin{proof}[Another proof of the necessarity] Alternatively, here we present another proof of the necessarity.
By making use of invariance property of $P_i$ and $P_i^2=P_i$, we have
\begin{eqnarray*}
Df^{n+m}(x)P_i(x)u &=& Df^{n+m}(x)P_i^2(x)Df^{-n}(f^n(x))Df^n(x)u\\[3mm]
&=& Df^{n+m}(x)P_i(x)Df^{-n}(f^n(x))P_i(f^n(x))Df^n(x)u \\[3mm]
&=& Df^{n+m}(x)P_i(x)Df^{-n}(f^n(x))Df^n(x)P_i(f(x))u.
\end{eqnarray*}

Similarly, we can get
\begin{equation*}
Df^n(x)P_j(x)v=Df^n(x)P_j(x)Df^{-(n+m)}(f^{n+m}(x))Df^{n+m}(x)P_j(x)v.
\end{equation*}

Therefore, by then definition of dominated splitting, we have
\begin{multline*}
Ce^{-m\alpha}\geq
\frac{\|Df^{n+m}(x)P_i(x)u\|\|Df^n(x)P_j(x)v\|}{\|Df^n(x)P_i(x)u\|\|Df^{n+m}(x)P_j(x)v\|} \\[3mm]
=\|Df^{n+m}(x)P_i(x)Df^{-n}(f^n(x))\tilde{u}\|
\|Df^n(x)P_j(x)Df^{-(n+m)}(f^{n+m}(x))\tilde{v}\|
\end{multline*}
where $\tilde{u}=\frac{Df^n(x)P_i(f(x))u}{\|Df^n(x)P_i(f(x))u\|}$ and
$\tilde{v}=\frac{Df^{n+m}(x)P_j(x)v}{\|Df^{n+m}(x)P_j(x)v\|}$. Hence,
\begin{multline*}
\|Df^{n+m}(x)P_i(x)Df^{-n}(f^n(x))\| \\
\|Df^n(x)P_j(x)Df^{-(n+m)}(f^{n+m}(x))\|\leq Ce^{-m\alpha}
\end{multline*}
Let $j=i+1$ in above formula, we get \eqref{eq1}.
\end{proof}
%

\section{Upper and lower functions}

Let $\{p(k)\}_{k\in\mathbb{Z}}$ be a series. For each fixed $N\in \mathbb{Z}^+$, denote
$$p_N(k)=\frac{1}{N}\sum_{j=0}^{N-1}p(k+j).$$

We then have following fact.
\begin{lemma}\label{lemma3}
Let $\{p(k)\}_{k\in\mathbb{Z}}$ be a bounded series with a uniform bound $\|p\|$. Then for $n$, $m\in\mathbb{N}$ with $m\geq 1$, we have
\be\label{s2eq3}
\left\|\sum_{k=n}^{n+m-1}\big(p(k) - p_N(k)\big)\right\|< \|p\|N
\ee
\end{lemma}
\begin{proof}
By direct computation, we have
\bea\label{s2eq4}
\left\|\sum_{k=n}^{n+m-1}\big(p(k) - p_N(k)\big)\right\| & = &\frac{1}{N}\left\|\sum_{k=n}^{n+m-1}\sum_{j=0}^{H-1}\left(p(k)-p(k+j)\right)\right\| \nonumber \\[3mm]
& = & \frac{1}{N}\left\|\sum_{j=0}^{N-1}\left(\sum_{k=n}^{n+m-1}p(k) - \sum_{k=n+j}^{n+m+j-1} p(k)\right)\right\| \nonumber \\[3mm]
& = & \frac{1}{N}\left\|\sum_{j=0}^{N-1}\left(\sum_{k=n}^{n+j-1}p(k) - \sum_{k=n+m}^{n+m+j-1} p(k)\right)\right\| \nonumber \\[3mm]
& \leq & \frac{2\|p\|}{N}\sum_{j=1}^{N-1}j < \|p\|N
\eea
\end{proof}

Next we introduce the notion of upper and lower functions.
\begin{definition}\label{upperlowerfunction}
Let $f: M\to M$ be a diffeomorphism on a closed manifold $M$ and $\Lambda\subset M$ be any compact $f$-invariant set. Let the splitting $T_xM=E_1(x)\oplus\cdots\oplus E_k(x)$, $x\in\Lambda$ of tangent bundle over $\Lambda$ be $Df$-invariant. A function $g:\Lambda\to\mathbb{R}$ is called an \emph{upper (}or \emph{lower)} \emph{function} of $E_i$ with respect to $f$ over $\Lambda$ if there exists a constant $K_i>1$ such that for $n,m\in\mathbb{Z}$, $m\geq 1$,
\begin{subequations}\label{s2eq5}
\be
\|Df^{n+m}(x)P_i(x)Df^{-n}(f^n(x))\|\leq K_i \exp{{\left(\sum_{k=n}^{n+m-1} g(f^k(x))\right)}};
\ee
\be
\left(\text{or }\|Df^{n}(x)P_i(x)Df^{-(n+m)}(f^{n+m}(x))\|\leq K_i \exp{\left(-{\sum_{k=n}^{n+m-1} g(f^k(x))}\right)}\right)
\ee
\end{subequations}
\end{definition}
For reducible splittings, we have the following lemma.
\begin{lemma}\label{lemma4}
Let the splitting $T_{\Lambda}M=E_1\oplus\cdots\oplus E_k$ be reducible, i.e. $\|Df^n(x)P_i(x)Df^{-n}(f^n(x))\|\leq K$ for all $n\in\mathbb{Z}$, $x\in\Lambda$ and $i=1,\cdots,k$ with $P_i(x)$ the corresponding projectors. Then the functions
\be\label{s2eq6}
\rho^{+}_{N,i}(x) = \frac{1}{N} \log{\|Df^{N}(x)P_i(x)\|}
\ee
and
\be\label{s2eq7}
\rho^{-}_{N,i-1}(x) = -\frac{1}{N} \log{\|P_i(x)Df^{-N}(f^N(x))\|}
\ee
with $N\in\mathbb{Z}^+$ and $x\in\Lambda$ are upper and lower functions of $E_i$ with respect to $f$ respectively.

\end{lemma}
\begin{proof}
For $u\in T_xM$ with $P_i(x)u\neq 0$, define for each $n\in\mathbb{Z}$,
\begin{equation}\label{s2eq1}
p_{n,x,u}(k) = \log{\frac{\|Df^{k+1}(f^n(x)) P_i(f^n(x)) u\|}{\|Df^{k}(f^n(x))P_i(f^n(x))u\|}},
\end{equation}
and denote
\begin{equation}\label{s2eq2}
p_{N,n,x,u}(k)=\frac{1}{N}\sum_{j=0}^{N-1}p_{n,x,u}(j+k),
\end{equation}
for each fixed $N\in\mathbb{Z}^+$. Clearly we have $p_{n,x,\xi}(k)\in [-\log{\|Df\|}, \log{\|Df\|}]$, for all $n$, $x$, $k$ and $u$.

By rewriting and using Lemma \ref{lemma3}, we get,
\bea\label{s2eq9}
& &\frac{\|Df^{m}(f^n(x))P_i(f^n(x))u\|}{\|P_i(f^n(x))u\|} \nonumber \\[3mm]
& = & \frac{\|Df^{m}(f^n(x))P_i(f^n(x))u\|}{\|Df^{m-1}(f^n(x))P_i(f^n(x)) u\|}\frac{\|Df^{m-1}(f^n(x))P_i(f^n(x))u\|}{\|Df^{m-2}(f^n(x))P_i(f^n(x)) u\|} \nonumber \\
&& \qquad \cdots\frac{\|Df(f^n(x)) P_i(f^n(x))u\|}{\|P_i(f^n(x))u\|} \nonumber \\[3mm]
& = & \exp{\left(\sum_{k=0}^{m-1}\log{\frac{\|Df^{k+1}(f^n(x))P_i(f^n(x)) u\|}{\|Df^{k}(f^n(x))P_i(f^n(x))u\|}}\right)} \nonumber \\[3mm]
& = & \exp{\left(\sum_{k=0}^{m-1} p_{n,x,\xi}(k)\right)} \nonumber \\[3mm]
& < & \exp{\left(N\log \|Df\|\right)}\exp{\left(\sum_{k=0}^{m-1} p_{N,n,x,u}(k)\right)}
\eea

Next we have the following estimate on $p_{N,n,x,u}(k)$.
\bea\label{s2eq10}
p_{N,n,x,u}(k) & = & \frac{1}{N}\sum_{j=0}^{N-1} p_{n,x,u}(j+k) \nonumber \\[3mm]
& = & \frac{1}{N} \log{\frac{\|Df^{k+N}(f^n(x))P_i(f^n(x)) u\|}{\|Df^{k}(f^n(x))P_i(f^n(x))u\|}} \nonumber\\[3mm]
& \leq & \frac{1}{N}\log{\|Df^{k+N}(f^n(x))P_i(f^n(x))Df^{-k}(f^{n+k}(x))\|} \nonumber \\[3mm]
& = & \rho^{+}_{N,i}(f^{k+n}(x)).
\eea

In the last equality, we used the $Df$-invariance of projector $P_i$, that is,
\[Df(x)P_i(x)=P_i(f(x))Df(x),\quad\text{for all } x\in\Lambda\]
and the unfolding $Df^{n+m}(x)=Df^m(f^{n}(x))Df^{n}(x)$, where $m,n\in\mathbb{Z}$.

Combine \eqref{s2eq9} and \eqref{s2eq10}, we readily have
\bea\label{s2eq12}
&&\|Df^{n+m}(x)P_i(x) Df^{-n}(f^n(x))u\|\nonumber\\[3mm]
&=& \|Df^{m}(f^n(x))P_i(f^n(x))Df^n(x)P_i(x) Df^{-n}(f^n(x))u\|\nonumber\\[3mm]
& \leq & \|Df^{n}(x) P_i(x) Df^{-n}(f^n(x))u\| \nonumber \\
&& \qquad \exp{(N\log\|Df\|)}\exp{\left(\sum_{k=0}^{m-1}\rho^{+}_{H,i}(f^{k+n}(x))\right)} \nonumber\\[3mm]
& \leq & K \exp{\left(\sum_{k=n}^{n+m-1}\rho^{+}_{H,i}(f^{k}(x))\right)\|u\|},
\eea
where the last step uses that the splitting is reducible and $K$ also depends on $N$. This shows that $\rho^{+}_{H,i}(\cdot)$ is an upper function of $E_i$ with respect to $f$. Similarly we can get $\rho^{-}_{H,i-1}(\cdot)$ is a lower function of $E_i$ with respect to $f$.
\end{proof}

Now we are ready to give another characterization of dominated splitting in terms of upper and lower functions constructed above.
\begin{theorem}
The splitting $T_{\Lambda}M=E_1\oplus\cdots\oplus E_k$ is $(n_1,\cdots,n_k)$-dominated if and only if $\rho^+_{N,i}$ and $\rho^-_{N,i}$, $N$ big enough, defined in Lemma \ref{lemma4} are supper and lower functions of $E_i$ and $E_{i+1}$ with respect to $f$ respectively and there exist a constant $\alpha>0$ such that for for all $x\in\Lambda$,
\begin{equation}\label{slope}
\rho^+_{N,i}(x) - \rho^-_{N,i}(x)\leq \frac{\alpha}{2},\quad i=1,\cdots,k-1.
\end{equation}
\end{theorem}

\begin{proof}
We first suppose that the splitting $T_{\Lambda}M=E_1\oplus\cdots\oplus E_k$ is $(n_1,\cdots,n_k)$-dominated. Let $P_i$, $i=1,\cdots,n$ be the corresponding projectors on $T_{\Lambda}M$. Then by Lemma \ref{lemma4}, we are left to show \eqref{slope}. In fact, by choosing $N$ big enough such that $N>2\alpha^{-1}\log{K}$, and a direct calculation shows
\begin{eqnarray*}
& & \rho^+_{N,i}(x) - \rho^-_{N,i}(x) \nonumber \\[3mm]
& = & \frac{1}{N}\log\{\|Df^{N}(x)P_i(x)\|\|P_{i+1}(x)Df^{-N}(f^{N}(x))\|\}\nonumber \\[3mm]
& \leq & \frac{1}{N} \log\{Ke^{-\alpha N}\} \leq -\frac{\alpha}{2}.
\end{eqnarray*}
In the first inequality, we used Theorem \ref{mainthm1} by letting $m=N$ and $n=0$.

Next we fix $i$ and suppose for some $N$, $\rho^+_{N,i}$ and $\rho^-_{N,i}$ are the supper and lower functions of $E_i$ and $E_{i+1}$ with respect to $f$ respectively and \eqref{slope} holds. By making use of the invariance property of projector $P_i$, i.e., $P_i(f^n(x))Df^n(x)=Df^n(x)P_i(x)$, $n\in\mathbb{Z}$ and $P_i^2=P_i$ one can show for $l\in\mathbb{Z}$, $l\geq 1$ that
\[\|Df^{lN}(x)P_i(x)\|\|P_{i+1}(x)Df^{-lN}(f^{lN}(x))\|\leq \exp\{-\frac{lN}{2}\alpha\}.\]
Indeed,
\begin{eqnarray*}
&&\|Df^{lN}(x)P_i(x)\|\|P_{i+1}(x)Df^{-lN}(f^{lN}(x))\|\\[3mm]
&=&\|Df^{N}(f^{(l-1)N}(x))P_i(f^{(l-1)N}(x))\cdots Df^N(f^N(x))P_{i}(f^N(x)) \nonumber \\
&& Df^N(x)P_i(x)\| \|P_{i+1}(x)Df^{-N}(f^N(x))P_{i+1}(f^N(x))Df^{-N}(f^{2N}(x)) \nonumber \\
&& \cdots P_{i+1}(f^{(l-1)N}(x))Df^{-N}(f^{lN}(x))\|\\[3mm]
&\leq & \prod_{k=0}^{l-1}\big(\|Df^N(f^{kN}(x))P_i(f^{kN}(x))\|\|P_{i+1}(f^{kN}(x))Df^{-N}(f^{kN}(x))\|\big)\\[3mm]
&= & \prod_{k=0}^{l-1}\exp{\{N(\rho^+_{N,i}(f^{kN}(x)) - \rho^-_{N,i}(f^{kN}(x)))\}}\\[3mm]
&\leq & \exp\{-\frac{lN}{2}\alpha\}.
\end{eqnarray*}
For any fixed $m,n\in\mathbb{Z}$, $m\geq 1$. Assume $m=lN+k$ for some $l,k\in\mathbb{Z}$, $0\leq k< N$ and denote $f^n(x)=y$. We have
\begin{eqnarray*}
&&\|Df^{n+m}(x)P_i(x)Df^{-n}(f^n(x))\|\|Df^n(x)P_{i+1}(x)Df^{-(m+n)}(f^{m+n}(x))\|\\[3mm]
&=&\|Df^{lN+k}(y)P_i(y)\|\|P_{i+1}(y)Df^{-(lN+k)}(f^{lN+k}(y))\|\\[3mm]
&=&\|Df^{lN}(f^k(y))P_i(f^k(y))Df^k(y)P_i(y)\| \nonumber \\
&& \qquad \|P_{i+1}(y)Df^{-k}(f^k(y))P_{i+1}(f^k(y))Df^{lN}(f^{lN}(f^k(y)))\|\\[3mm]
&\leq & \|Df^{lN}(f^k(y))P_i(f^k(y))\|\|P_{i+1}(f^k(y))Df^{-lN}(f^{lN}(f^k(y)))\|\\
&& \qquad \|Df^k(y)P_i(y)\|\|P_{i+1}(y)Df^{-k}(f^k(y))\|\\[3mm]
&\leq & \tilde{C}\exp\{-\frac{\alpha}{2}(m-k)\} \leq  C\exp\{-\frac{\alpha}{2}m\}.
\end{eqnarray*}

Note that above constants $C$ and $\alpha$ do not dependent on the choice of $i$, therefore by Theorem \ref{mainthm1}, the splitting $T_{\Lambda}M=E_1\oplus\cdots\oplus E_k$ is dominated.
\end{proof}

\section{Summerably Separation}

In this section we introduce the notion of summably separated functions with respect to a dynamical system and use it to characterize dominated splitting. It also states that after a functional torsion to the dominated splitting system it becomes to hyperbolic.

\begin{definition}
Let $f: M\to M$ be a diffeomorphism on a closed manifold $M$ and $\Lambda\subset M$ be any compact $f$-invariant set. A sequence of real continuous functions $p_i: \Lambda\to\mathbb{R}$, $i=1,2,\cdots,k$ are called \emph{summably separated} with respect to $f$ on $\Lambda$ if there exist $\beta\geq 0$, $\gamma>0$, such that for $i=1,\cdots,k-1$,
\be\label{s3eq1}
\sum_{k=n}^{n+m-1}\left(p_{i+1}(f^k(x)) - p_i(f^k(x))\right) \geq -\beta+\gamma m,
\ee
for all $m,n\in\mathbb{Z}^+$, $m\geq 1$ and $x\in\Lambda$.
\end{definition}

\begin{theorem}
Let $f: M\to M$ be a diffeomorphism on a closed manifold $M$ and $\Lambda\subset M$ be any compact $f$-invariant set. A splitting $T_{\Lambda}M = E_1 \oplus\cdots\oplus E_{k}$ of tangent bundle over $\Lambda$ is $(n_1, \cdots, n_k)$-dominated if and only if there exist continuous real functions $p_i: \Lambda\to\mathbb{R}^{+}$, $i=1,\cdots,k$, with $\log p_1,\cdots,\log p_k$ are summably separated with respect to $f$, such that for $i=1,\cdots,k$ the linear cocycle $(f,p_iDf)$ admits a hyperbolicity over $\Lambda$ with stable subspace of dimension $n_1+\cdots+n_i$.
\end{theorem}

\begin{proof}
We first suppose that the splitting $T_{\Lambda}M=E_1\oplus\cdots\oplus E_k$ is $(n_1,\cdots,n_k)$-dominated. Let $P_i$, $i=1,\cdots,n$ be the corresponding projectors on $T_{\Lambda}M$. Then by Theorem \ref{mainthm1}, there exist constants $K>1$ and $\alpha>0$ such that \eqref{eq1} holds. In addition, by the reducible property of dominated splitting (Lemma \ref{lemma2}) and Lemma \ref{lemma4}, $\rho_{N,i}^+(\cdot)$ and $\rho_{N,i-1}^-$ (see \eqref{s2eq6}, \eqref{s2eq7}) are the upper and lower functions of $E_i$ with respect to $f$ respectively. That is, one can finds constants $K_i> 1$, $i=1,\cdots,k$ such that \eqref{s2eq5} hold.
Choose $N$ with $N>2\alpha^{-1}\log{K}$, then we have the following estimate of the difference between an upper and a lower function
\bea\label{s3eq4}
& & \rho^+_{N,i}(f^k(x)) - \rho^-_{N,i}(f^k(x)) \nonumber \\[3mm]
& = & \frac{1}{N}\log\{\|Df^{k+N}(x)P_i(x) Df^{-k}(f^k(x))\| \nonumber \\
&& \qquad \|Df^{k}(x) P_{i+1}(x)Df^{-(k+N)}(f^{k+N}(x))\|\}\nonumber \\[3mm]
& \leq & \frac{1}{N} \log\{Ke^{-\alpha N}\} \leq -\frac{\alpha}{2}.
\eea
In the first inequality, we used Theorem \ref{mainthm1}.

On the other hand, we have for all $n, m\in\mathbb{Z}$ with $m\geq 1$,
\bea\label{s3eq5}
1 & \leq & \|Df^{n+m}(x)P_i(x)Df^{-(n+m)}(f^{n+m}(x))\| \nonumber\\[3mm]
& \leq & \|Df^{n+m}(x)P_i(x)Df^{-n}(f^n(x))\| \|Df^{n}(x)P_i(x) Df^{-(n+m)}(f^{n+m}(x))\| \nonumber\\[3mm]
&\leq & K_i^2 \exp{\left(\sum_{k=n}^{n+m-1}[\rho^+_{N,i}(f^k(x)) - \rho^-_{N,i-1}(f^k(x))]\right)},\nonumber
\eea
which implies
\be\label{s3eq6}
\sum_{k=n}^{n+m-1}[\rho^+_{N,i}(f^k(x)) - \rho^-_{N,i-1}(f^k(x))] \geq -2\log{K_i}
\ee

Combine \eqref{s3eq4} and \eqref{s3eq6}, we get an estimate between two adjacent upper functions
\bea\label{s3eq7}
& & \sum_{k=n}^{n+m-1}[\rho^+_{N,i}(f^k(x)) - \rho^+_{N,i-1}(f^k(x))]  \nonumber\\[3mm]
& = & \sum_{k=n}^{n+m-1}\left([\rho^+_{N,i}(f^k(x)) - \rho^-_{N,i-1}(f^k(x))]- \right.\nonumber\\
&& \left.[\rho^+_{N,i-1}(f^k(x))-\rho^-_{N,i-1}(f^k(x))]\right)\nonumber \\[3mm]
& \geq & -2\log{K_i} + \frac{\alpha}{2}m.
\eea

We choose $p_i(x)$ to have the form
\be\label{s3eq12}
p_i(x) = \exp{\left(-\rho^+_{N,i}(x)-\lambda\right)}, \ \lambda>0.
\ee

Since $\rho^+_{N,i}(\cdot)$ is continuous on $\Lambda$ and $\rho^+_{N,i}(x)\in [-\log\|Df\|,\log\|Df\|]$ for all $x\in\Lambda$ and $i=1,\cdots, k$. \eqref{s3eq12} implies that functions $p_1,\cdots,p_k$ are summablely separated with respect to $f$.

Moreover, we have
\bea\label{s3eq8}
& & \|Df^{n+m}(x)(P_1(x)+\cdots+P_i(x))Df^{-n}(f^n(x))\| \nonumber \\[3mm]
& \leq & \sum_{j=1}^{i}\|Df^{n+m}(x) P_j(x)Df^{-n}(f^n(x))\| \nonumber \\[3mm]
& \leq & (K_i+K_{i-1}K_i^2+\cdots+K_1K_2^2\cdots K_i^2) \exp{\left(\sum_{k=n}^{n+m-1}\rho^+_{N,i}(f^k(x))\right)} \nonumber\\[3mm]
& = & M_i \exp{\left(\sum_{k=n}^{n+m-1}\rho^+_{N,i}(f^k(x))\right)}
\eea

By a similar argument, we can also get
\begin{multline}\label{s3eq9}
\|Df^{n}(x)(P_{i+1}(x)+\cdots+P_k(x))Df^{-(n+m)}(f^{n+m}(x))\| \\[3mm]
\leq N_i \exp{\left(-\sum_{k=n}^{n+m-1}\rho^-_{N,i}(f^k(x))\right)}
\end{multline}

Therefore for the cocycle generated by $A_i(x)=p_i(x)Df(x)$, with
\bea\label{s3eq10}
A_i^n(x)
& = & p_i(f^{n-1}(x))Df(f^{n-1}(x)) p_i(f^{n-2}(x))Df(f^{n-2}(x)) \nonumber\\
&& \qquad\cdots  p_i(x)Df(x) \nonumber\\[3mm]
& = & \left(\prod_{k=0}^{n-1}p_i(f^k(x))\right)Df^n(x),
\eea
\bea\label{s3eq22}
A_i^{-m}(x) & = & p_i^{-1}(f(x))Df^{-1}(f(x)) p_i^{-1}(f^{2}(x))Df^{-1}(f^{2}(x)) \nonumber \\
&& \qquad\cdots p_i^{-1}(f^m(x))Df^{-1}(f^m(x)) \nonumber\\[3mm]
& = & \left(\prod_{k=1}^{m}p_i^{-1}(f^k(x))\right)Df^{-m}(x),
\eea
where $p_i^{-1}(f^k(x))=\frac{1}{p_i(f^{k-1}(x))}$, we have
\bea\label{s3eq11}
& & \|A_i^{n+m}(x)(P_1(x)+\cdots+P_i(x))A_i^{-n}(f^n(x))\| \nonumber \\[3mm]
& = & \bigg{\|}\left(\prod_{k=0}^{n+m-1}p_i(f^k(x))\right)Df^{n+m}(x) \left(P_1(x)+\cdots+P_i(x)\right) \nonumber \\
&& \qquad \left(\prod_{k=1}^{n}p_i^{-1}(f^k(x))\right)Df^{-n}(f^n(x))\bigg\| \nonumber \\[3mm]
& = & \prod_{k=n}^{n+m-1}p_i(f^k(x))\|Df^{n+m}(x)\left(P_1(x)+\cdots+P_i(x)\right) Df^{-n}(f^n(x)) \|\nonumber \\[3mm]
& \leq & M_i\prod_{k=n}^{n+m-1}p_i(f^k(x)) \exp{\left(\sum_{k=n}^{n+m-1} \rho^+_{H,i}(f^k(x))\right)} \quad (\mbox{by} \  (\ref{s3eq8}))
\eea

Then  \eqref{s3eq11} implies from \eqref{s3eq12} that
\be\label{s3eq13}
\|A_i^{n+m}(x)(P_1(x)+\cdots+P_i(x)) A_i^{-n}(f^n(x))\|\leq M_i e^{-m\lambda}.
\ee
On the other hand,
\bea\label{s3eq14}
& & \|A_i^{n}(x)(P_{i+1}(x)+\cdots+P_k(x)) A_i^{-(n+m)}(f^{n+m}(x))\| \nonumber \\[3mm]
& = & \bigg\|\left(\prod_{k=0}^{n-1}p_i(f^k(x))\right)Df^{n}(x) \left(P_{i+1}(x)+\cdots+P_k(x)\right) \nonumber \\
&& \qquad\left(\prod_{k=1}^{n+m}p_i^{-1}(f^k(x))\right)Df^{-(n+m)}(f^{n+m}(x))\bigg\| \nonumber \\[3mm]
& = & \prod_{k=n+1}^{n+m}p_i^{-1}(f^k(x))\|Df^{n}(x)\left(P_{i+1}(x)+\cdots+P_k(x)\right) \nonumber \\
&& \qquad Df^{-(n+m)}(f^{n+m}(x)) \|\nonumber \\[3mm]
& \leq & N_i \exp{\left(-\sum_{k=n}^{n+m-1} [\rho^-_{N,i}(f^k(x))-\rho^+_{N,i}(f^k(x))-\lambda]\right)} 
\nonumber\\[3mm]
& \leq & N_i e^{-m(\frac{\alpha}{2}-\lambda)}
\eea

Therefore from \eqref{s3eq13}, \eqref{s3eq14}, if we constraint that $0<\lambda<\frac{\alpha}{2}$, the linear cocycle $(f,A_i)=(f,p_iDf)$ admit a hyperbolicity over $\Lambda$. Moreover, the dimension of its stable manifold is $\dim(P_1+\cdots+P_i)=n_1+\cdots+n_i$.

Now we assume that  there exist positive bounded real continuous functions $p_i(x)$, $i=1,\cdots,k$ on $\Lambda$ with $\log p_1,\cdots,\log p_i$ summablely separated, such that each linear cocycle $(f,p_iDf)$ admits hyperbolicity over $\Lambda$ with corresponding projectors $Q_i$ and stable subbundle $Q_i(T_{\lambda}M)$ of dimension $n_1+n_2+\cdots+n_i$. By hyperbolicity, one can finds constants $C>1$ and $\alpha>0$ such that for $i=1,\cdots,k$,
\begin{eqnarray}
&&\|A_i^{n+m}(x) Q_i(x)A_i^{-n}(f^n(x))\|\leq Ce^{-m\alpha},\label{s3eq23}\\[3mm] &&\|A_i^{n}(x)(I-Q_i(x))A_i^{-(n+m)}(f^{n+m}(x))\|\leq Ce^{-m\alpha},\label{s3eq15}
\end{eqnarray}
for all $x\in\Lambda$ and $n,m\in\mathbb{Z}$, $m\geq 1$. Here $A^n_i$ and $A^{-m}_i$ are defined as \eqref{s3eq10} and \eqref{s3eq22} respectively.

Since $\log p_1,\cdots,\log p_k$ are summablely separated, it is not hard to see that $\text{range}(Q_1)\subsetneqq\cdots\subsetneqq \text{range}(Q_k)$. Then there exists unique supplementary invariant projectors $P_i$, $i=1,\cdots,k$, on $T_{\Lambda}M$ such that $\dim P_i=n_i$ and $Q_i=P_1+\cdots+P_i$, $I-Q_i=P_{i+1}+\cdots+P_k$.

We claim:
\be\label{s3eq16}
\|A_i^{n+m}(x)P_i(x)A_i^{-n}(f^n(x))\|\leq 2C^2e^{-m\alpha}.
\ee

To see this, first we have by using $P_i(x)=Q_i(x)-Q_{i-1}(x)$ that
\begin{multline}\label{s3eq17}
\|A_i^{n+m}(x) P_i(x)A_i^{-(n+m)}(f^{n+m}(x))\| \leq \|A_i^{n+m}(x)Q_i(x) \\[3mm]
A_i^{-(n+m)}(f^{n+m}(x))\|+\|A_i^{n+m}(x) Q_{i-1}(x)A_i^{-(n+m)}(f^{n+m}(x))\|
\end{multline}

The first term is clearly bounded by constant $C$ by letting $m=0$ in \eqref{s3eq23}. For the second term, we have
\bea\label{s3eq18}
& & \|A_i^{n+m}(x)Q_{i-1}(x)A_i^{-(n+m)}(f^{n+m}(x))\| \nonumber \\[3mm]
& = & \|\prod_{k=0}^{n+m-1}p_i(f^k(x))Df^{n+m}(x)Q_{i-1}(x) Df^{-(n+m)}(f^{n+m}(x))\nonumber \\
&& \qquad \prod_{k=1}^{n+m}p_{i}^{-1}(f^k(x)) \| \nonumber\\[3mm]
& = & \|Df^{n+m}(x)Q_{i-1}(x)Df^{-(n+m)}(f^{n+m}(x))\| \nonumber \\[3mm]
& = & \|\prod_{k=0}^{n+m-1}p_{i-1}(f^k(x))Df^{n+m}(x)Q_{i-1}(x) Df^{-(n+m)}(f^{n+m}(x))\nonumber \\
&&\qquad \prod_{k=1}^{n+m}p_{i-1}^{-1}(f^k(x)) \|\nonumber \\[3mm]
& = & \|A_{i-1}^{n+m}(x) Q_{i-1}(x) A_{i-1}^{-(n+m)}(f^{n+m}(x))\| \leq C
\eea
where we let again $m=0$ in \eqref{s3eq23} and use the definitions of $A_i^{n+m}(x)$ and $A_i^{-(n+m)}(f^{n+m}(x))$. We thus get, by also using $P_i(x)=P_i(x)Q_i(x)$,
\bea\label{s3eq19}
& &\|A_i^{n+m}(x)P_i(x)A_i^{-n}(f^n(x))\| \nonumber\\[3mm]
& = & \|A_i^{n+m}(x)P_i(x)Q_i(x)A_i^{-n}(f^n(x))\|\nonumber\\[3mm]
& \leq & \|A_i^{n+m}(x)P_i(x)A_i^{-(n+m)}(f^{n+m}(x))\| \|A_i^{n+m}(x) Q_i(x)A_i^{-n}(f^n(x))\|\nonumber\\[3mm]
& \leq & 2C\cdot Ce^{-m\alpha} = 2C^2 e^{-m\alpha}.
\eea

Likewise, we can also get
\be\label{s3eq20}
\|A_i^{n}(x)P_{i+1}(x)A_{i}^{-(n+m)}(f^{n+m}(x))\| \leq 2C^2 e^{-m\alpha}
\ee

Making use of \eqref{s3eq19} and \eqref{s3eq20}, we readily have
\bea\label{s3eq21}
& & \|Df^{n+m}(x)P_i(x)Df^{-n}(f^n(x))\|\|Df^{n}(x) P_{i+1}(x) Df^{-(n+m)}(f^{n+m}(x))\| \nonumber\\[3mm]
& = & \|A_i^{n+m}(x)P_i(x)A_i^{-n}(f^n(x))\| \|A_{i}^{n}(x) P_{i+1}(x)A_{i}^{-(n+m)}(f^{n+m}(x))\|\nonumber\\[3mm]
&\leq & 4C^4e^{-2m\alpha},\nonumber
\eea
which implies, by Theorem \ref{mainthm1}, that the splitting $T_{x}M=E_1(x)\oplus\cdots\oplus E_k(x)$, $x\in\Lambda$ of tangent bundle over $\Lambda$, where $E_i(x)=\text{range}(P_i(x))$, is $(n_1, n_2, \cdots, n_k)$-dominated. This completes the proof of the theorem.
\end{proof}

\section{Discussion and Question}
Let $\Lambda$ be a subset of $M$. We denote by $\text{Diff}^1_T(\Lambda)$ the space of $C^1$ diffeomorphisms on $M$ which make $\Lambda$ to be a transitive set. Moreover, we denote $\mathcal{DS}$ the set of systems admitting dominated splitting over $\Lambda$ in $\text{Diff}^1_T(\Lambda)$. Let $\mathcal{RS}$ denote the reducible systems in $\text{Diff}^1_T(\Lambda)$. Then it follows from roughness of dominated splitting systems and the fact that dominated splitting implies reducible that $\mathcal{DS}\subset\text{int}(\mathcal{RS})$.

\begin{description}
\item[Qeustion]
Does $\text{int}(\mathcal{RS})=\mathcal{DS}$?
\end{description}

Generally, if there is no constriction for diffeomorphisms over $\Lambda$, the answer may be negative. To explain this, let take the most simplest case $\Lambda=\{x_1,x_2\}$ and two subbundles' case for example. Then one can find a diffeomorphism $f$ such that $x_1$ and $x_2$ are two of fixed hyperbolic points of $f$ with respect to the decomposition $T_{x_1}M=E_1\oplus F_1$ and $T_{x_2}M=E_2\oplus F_2$ respectively. Moreover, we can require $\dim(E_1)=\dim(F_2)\neq\dim(E_2)$. Then we observe that the functions in the neighborhood of $f$ in $\text{Diff}^1(\Lambda)$ are all satisfied these conditions, which means $f\in\text{int}\mathcal{RS}$. However, it is almost obvious that $f\notin\mathcal{DS}$.\\[3mm]

\textbf{Acknowledgements} The authors would like to thank Gang Liao, Enriuqe Ramiro Pujals and Sylvain Crovisier for discussions and useful comments.



\end{document}